\def\version{Version 2 -- last updated 03/04/2026
\hfill\href{https://arxiv.org/abs/?}{arXiv:2603.29425}
}

\documentclass[11pt,a4paper]{amsart}
\usepackage{fullpage}

\usepackage[hypertexnames=false]{hyperref}

\usepackage{amssymb,amscd,amsxtra}
\usepackage{autobreak}

\usepackage{moremath}

\usepackage[alphabetic,lite]{amsrefs}

\usepackage[autobold]{mathfixs}

\usepackage[scr,scaled=1.1]{rsfso}

\usepackage{leftidx}

\usepackage[clockwise]{rotating}

\usepackage{xkvltxp}
\usepackage[nomargin,inline]{fixme}
\fxusetheme{color}
\definecolor{fxnote}{rgb}{0.0000,0.6000,0.0000}
\usepackage[]{todonotes}

\usepackage{color-edits}
\addauthor[AB]{AB}{magenta}

\usepackage{mathtools}
\usepackage{relsize}

\usepackage{tikz}
\usetikzlibrary{arrows,calc}

\usepackage{tikz-cd}
\usepackage{tikz-network}

\usepackage{xcolor}
\usepackage[all,poly,web,color,matrix,arrow]{xy}
\newdir{ >}{{}*!/-8pt/@{>}}

\usepackage{braket}

\newtheorem{thm}{Theorem}[section]

\newtheorem{prop}[thm]{Proposition}

\theoremstyle{definition}
\newtheorem{rem}[thm]{Remark}
\newtheorem*{rem*}{Remark}
\newtheorem{defn}[thm]{Definition}

\newtheorem{conj}[thm]{Conjecture}

\numberwithin{equation}{section}
\numberwithin{figure}{section}

\def\:{\colon}
\def\.{\cdot}

\def\<{\left\langle}
\def\>{\right\rangle}
\def\({\left(}
\def\){\right)}

\def\epsilon{\varepsilon}
\def\phi{\varphi}

\def\leq{\leqslant}
\def\geq{\geqslant}

\def\Lra{\Longrightarrow}

\def\bar#1{\overline{#1}}

\def\tilde#1{\widetilde{#1}}
\def\iso{\cong}

\def\F{\mathbb{F}}

\def\Z{\mathbb{Z}}

\DeclareMathOperator{\Ext}{Ext}

\DeclareMathOperator{\Hom}{Hom}

\def\O{\mathrm{O}}
\def\SO{\mathrm{SO}}

\def\SU{\mathrm{SU}}
\def\String{\mathrm{String}}

\def\kO{{k\mathrm{O}}}

\DeclareMathOperator{\Sq}{Sq}

\def\dlQ{\mathrm{Q}}

\def\QS0{\dlQ S^0}
\def\QSo0{\dlQ_0S^0}
\def\StA{\mathcal{A}}

\def\StP{\mathcal{P}}

\def\PL{\mathrm{PL}}
\def\TOP{\mathrm{TOP}}
\def\G{\mathrm{G}}

\title[Poincar\'e duality spaces related to the Joker]
{Poincar\'e duality spaces related to the Joker}
\author{Andrew Baker}
\date{\version}
\address{
School of Mathematics \& Statistics,
University of Glasgow, Glasgow G12~8QQ, Scotland.}
\email{andrew.j.baker@glasgow.ac.uk}
\urladdr{http://www.maths.gla.ac.uk/$\sim$ajb}
\thanks{Thanks to Diarmuid Crowley and Csaba Nagy 
for helpful discussions, and to John Rognes for
pointing out a relevant reference. Don Davis has 
also recently drawn attention to the homogeneous 
space discussed in Section~\ref{sec:Manifold}.}

\begin{document}

\begin{abstract}
The well known Joker $\StA(1)$-module of Adams 
\& Priddy is known to be realisable as the 
cohomology of a $1$-connected space. By attaching 
an extra cell we obtain an $8$-dimensional 
Poincar\'e duality space whose mod~$2$ cohomology 
realising is an unstable $\StA$-algebra. We 
use obstruction theory to show that this admits 
a $\PL$-structure. Although we are unable to 
show it is smoothable, it turns out that the 
cohomology can be realised as that of a 
homogeneous space.
\end{abstract}

\maketitle


\section*{Introduction}

In this note we investigate an $8$-dimensional
Poincar\'e duality space whose mod~$2$ cohomology 
as a Steenrod module algebra is related to the 
Joker module of Adams \& Priddy~\cite{JFA&SBP}. 
It turns out that such a PD space always admits 
a $\PL$-structure and its cohomology can be 
realised as that of a homogeneous space although 
it is unclear if this has the same homotopy type. 

\bigskip
\noindent
\textbf{Some history:} The author was led 
to this example by serendipity. While reading 
about Poincar\'e duality in Larry Smith's 
book \cite{LS:PolyInvtsFinGps-Book}, he 
looked for familiar examples and the 
natural symmetry of the Joker module 
suggested it might be part of one and
this fitted well with earlier work. The 
difficulty of passing from a $\PL$ to a 
smooth structure was overcome when Csaba 
Nagy pointed out some old calculations 
of Borel and Hirzebruch. We have decided 
to make this work available since it 
seems to provide an interesting example 
of surgery theory in action building on 
some known homotopy theory.

\bigskip
\noindent
\textbf{Notation, etc:} We will denote 
the unreduced cohomology of a space $X$
by $H^*(X)$, and also write $H^*(Y)$
for the cohomology of a spectrum. These 
notations are of course inconsistent:
for a space with disjoint basepoint
$X_+$, 
\[
H^*(X) \iso H^*(\Sigma^\infty(X_+)),
\]
however the context is usually sufficient
to make clear what is meant in practise.

\section{Poincar\'e duality algebras over 
the Steenrod algebra}\label{sec:SPD}

This section is based on work of Adams and 
Smith~\cites{JFA:FormulaeThom&Wu,LS:MilnorDiagElt}.

In this section we give some algebraic 
results on Poincar\'e duality algebras 
over the mod~$2$ Steenrod algebra~$\StA$. 
We will refer to such an algebra as a 
\emph{Steenrod Poincar\'e duality algebra} 
(SPDA). Although we work exclusively in   
characteristic~$2$, of course similar 
results hold for other primes.

We recall that for a left $\StA$-module 
$M=M^*$ its dual~$M_*$ (with grading 
defined by $M_n=\Hom_{\F_2}(M^n,\F_2)$ 
which we view as cohomological degree 
$-n$) becomes a right $\StA$-module 
so that for $x\in M$,
\[
(f\cdot\Sq^r)(x) = f(\Sq^rx).
\]
This means that the duality pairing 
$M_*\otimes M^*\to\F_2$ factors through
$M_*\otimes_{\StA}M^*$. Using the antipode 
$\chi$ of $\StA$ can also view~$M_*$ as 
a left $\StA$-module by setting
\[
\Sq^r\cdot f = f\cdot(\chi\Sq^r).
\]

If $P^*$ is an algebra then there is 
a pairing 
\[
\cap\:P_*\otimes P^*\to P_*;
\quad
f\cap a = f(a(-))
\]
making $P_*$ a right $P^*$-module and 
also a left $P^*$-module since $P^*$ is 
commutative. Taking the diagonal left 
$\StA$-module structure this is an 
$\StA$-module homomorphism.

One way to handle the interaction of 
the $P$ and $\StA$ module structures 
is by introducing the cross product 
algebra $P\sharp\StA$ which is the 
vector space $P\otimes\StA$ equipped 
with the multiplication that is given 
on basic tensors by
\[
(a\otimes\alpha)(b\otimes\beta)
= \sum_i a(\alpha'_i b)\otimes\alpha''_i\beta,
\]
where the coproduct on $\alpha$ is 
$\sum_i\alpha'_i\otimes\alpha''_i$.
We usually denote the element 
$a\otimes\alpha$ by $a\alpha$ so the 
above formula becomes
\begin{equation}\label{eq:CrosProd}
(a\alpha)(b\beta) 
= \sum_i a(\alpha'_i b)\alpha''_i\beta,
\end{equation}
It is also useful to record the easily 
proved formula 
\begin{equation}\label{eq:CrosProd-reverse}
a\alpha = \sum_i (1\alpha'_i)(\chi\alpha''_ia). 
\end{equation}
Then $P^*$ and $P_*$ are left $P\sharp\StA$-modules 
where for $x\in P^*$ and $f\in P_*$,
\[
(a\alpha)\cdot x = a\alpha(x),
\quad
(a\alpha)\cdot f = a\cdot(\alpha\cdot f).
\]

\begin{defn}\label{defn:SPDA}
A \emph{Steenrod Poincar\'e duality algebra 
of degree~$d$} is a finite type graded 
connected unstable commutative $\StA$-algebra 
$P=P^*$ with the following properties.
\begin{itemize}
\item
$P$ is a Poincar\'e duality algebra of 
degree~$d$: there is an element $[P]\in P_d$ 
such that $[P]\cap\:P^*\to P_{d-*}$ is 
an isomorphism of $P^*$-modules.
\item 
$[P]\cap$ is an isomorphism of left 
$P\sharp\StA$-modules.
\end{itemize}
\end{defn}

Since $[P]\cap$ is an isomorphism of 
$P^*$-modules, it follows that $P_*$ is 
a free $P^*$-module with basis~$[P]$.

Given such an SPDA~$P$ we can define 
some useful sequences of elements. 
\begin{defn}\label{defn:Wu&SW} 
The following sequences in $P^*$ are 
defined.
\begin{itemize}
\item 
The \emph{Wu classes}: $v_k\in P^k$ is 
the unique element for which 
\[
(\chi\Sq^k)\cdot[P]=[P]\cdot\Sq^k = [P]\cap v_k.
\]
\item 
The \emph{Stiefel-Whitney classes}:  
$w_k\in P^k$ is the unique element for 
which 
\[
w_k = \sum_{0\leq i\leq k}\Sq^iv_{k-i}.
\]
\item 
The \emph{dual Stiefel-Whitney classes}:  
$\bar{w}_k\in P^k$ is the unique element 
for which 
\[
\Sq^k\cdot[P]=[P]\cdot\chi\Sq^k = [P]\cap\bar{w}_k.
\]
\end{itemize}
We will refer to these elements as the
\emph{characteristic classes\/} of the 
SPDA. 
\end{defn}
Clearly we have $v_0=w_0=\bar{w}_0=1$.
\begin{prop}\label{prop:CharClasses}
The characteristic classes of the SPDA 
$P^*$ have the following properties. \\
\emph{(a)} 
If $k>\lfloor d/2\rfloor$ then $v_k=0$. \\
\emph{(b)} 
If $d$ is even then $w_d=v_{d/2}^2$.  \\
\emph{(c)}
For $1\leq k\leq d$ the Stiefel-Whitney 
classes satisfy
\[
\sum_{0\leq i\leq k}w_i\bar{w}_{k-i} = 0.
\]
\end{prop}
\begin{proof}
(a) This follows from the unstable algebra
condition. \\
(b) By definition and the unstable algebra 
condition,
\[
w_d = \sum_{0\leq i\leq d/2}\Sq^{d-i}v_i 
= \Sq^{d/2}v_{d/2} = v_{d/2}^2.
\]
(c) For $1\leq k\leq d$, 
\begin{align*}
\sum_{0\leq i\leq k}w_i\bar{w}_{k-i}\cdot[P] 
&= \sum_{0\leq i\leq k}w_i\Sq^{k-i}\cdot[P]  \\ 
&= \sum_{\substack{0\leq i\leq k\\ 0\leq j\leq i}}(\Sq^jv_{i-j})\Sq^{k-i}\cdot[P]  \\ 
&= \sum_{\substack{0\leq r\leq k\\ 0\leq j\leq k-r}}(\Sq^jv_r)\Sq^{k-r-j}\cdot[P]  \\ 
&= \sum_{\substack{0\leq r\leq k\\ 0\leq j\leq k-r}}(\Sq^jv_r)\Sq^{k-r-j}\cdot[P]  \\ 
&= \sum_{0\leq r\leq k}\Sq^{k-r}\cdot(v_r\cdot[P])  \\ 
&= \sum_{0\leq r\leq k}\Sq^{k-r}\cdot(\chi\Sq^r\cdot[P])  \\
&= \bigl(\sum_{0\leq r\leq k}\Sq^{k-r}\chi\Sq^r\bigr)\cdot[P] \\
&=0.
\qedhere
\end{align*}
\end{proof}

\begin{rem}\label{rem:WuCalc}
In general, calculation of the universal 
Wu classes for a manifold in terms of the 
universal Stiefel-Whitney classes can be 
non-trivial because of the involvement 
of the Steenrod algebra antipode. An 
alternative approach due to Atiyah \& 
Hirzebruch~\cite{MFA&FH:CohomOpCharKlass} 
involves the Todd polynomials; see Hirzebruch 
et al~\cite{HBJ:MfdsModForms}*{section~8.1}.
\end{rem}

We record another useful result.
\begin{prop}\label{prop:PDA-map}
Suppose that $P,Q$ are two Poincar\'e duality 
algebras of degree~$d$ and that $h\:P\to Q$ 
is an algebra homomorphism restricting to 
an isomorphism $h\:P^d\xrightarrow{\iso}Q^d$.
Then $h$ is injective.
\end{prop}
\begin{proof}
The adjoint map $h_*\:Q_*\to P_*$ gives an 
isomorphism $h_*\:Q_d\to P_d$ so we might 
as well assume that $h_*[Q]=[P]$ where 
$h_*[Q]=[Q]\circ h$. If $h\in P$ satisfies 
$h(x)=0$ then  
\[
[P]\cap x = [Q]\circ h(x) = 0.
\]
By freeness of the $P$-module $P_*$, this 
gives~$x=0$.
\end{proof}

\section{Some Poincar\'e duality complexes}\label{sec:PDcplxs}

In \cites{AB:Joker,AB&TB:Joker} we showed 
that the classical Joker $\StA(1)$-module 
and two of its iterated doubles could be 
realised as spaces and spectra. For example,
there is a CW complex with the following 
mod~$2$ cohomology as a module over~$\StA$,
where we show all the non-trivial actions
of $\Sq^{2^s}$ that occur and each $\bullet$ 
denotes a copy of~$\F_2$.

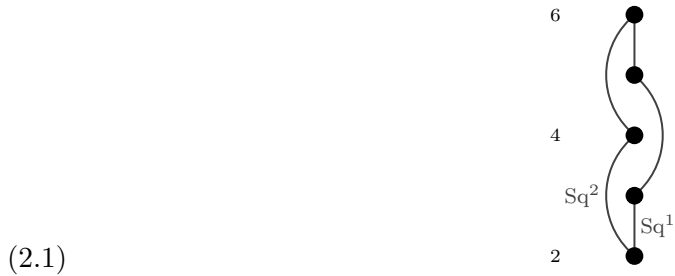
\begin{equation}\label{eq:H^*J}
\begin{tikzpicture}[scale=0.8]
\Vertex[y=6,size=.05,color=black]{A6}
\Text[x=-1,y=6,position=left,distance=1mm]{\tiny$6$}
\Vertex[y=5,size=.05,color=black]{A5}
\Vertex[y=4,size=.05,color=black]{A4}
\Text[x=-1,y=4,position=left,distance=1mm]{\tiny$4$}
\Vertex[y=3,size=.05,color=black]{A3}

\Vertex[y=2,size=.05,color=black]{A2}
\Text[x=-1,y=2,position=left,distance=1mm]{\tiny$2$}

\Edge[lw=0.75pt,bend=45](A4)(A6)
\Edge[lw=0.75pt](A5)(A6)
\Edge[lw=0.75pt,bend=-45](A3)(A5)
\Edge[lw=0.75pt,bend=45,label={$\Sq^2$},position=left](A2)(A4)
\Edge[lw=0.75pt,label={$\Sq^1$},position=right](A2)(A3)
\end{tikzpicture}
\end{equation}
The suspension spectrum of this complex is 
unique up to homotopy at least $2$-locally.

We will show that an additional $8$-cell can 
be added to get a CW complex whose cohomology
as an algebra over the Steenrod algebra is 
shown below, where the generator~$u_i$ has 
in degree~$i$. Here the $0$-cell is of course 
a disjoint basepoint.
\bigskip
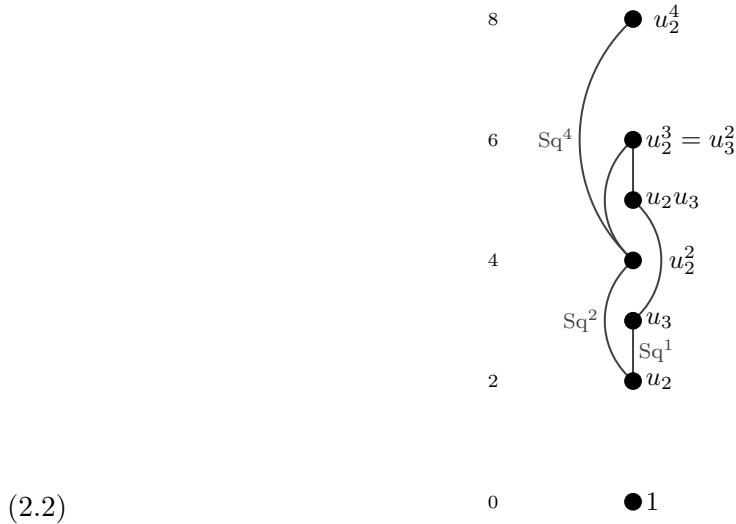
\begin{equation}\label{eq:H^*J8}
\begin{tikzpicture}[scale=0.8]
\Vertex[y=8,size=.05,color=black]{A8}
\Text[x=-2,y=8,position=left,distance=1mm]{\tiny$8$}
\Text[y=8,position=right,distance=2mm]{\small$u_2^4$}
\Vertex[y=6,size=.05,color=black]{A6}
\Text[x=-2,y=6,position=left,distance=1mm]{\tiny$6$}
\Text[y=6,position=right,distance=1mm]{\small$u_2^3=u_3^2$}
\Vertex[y=5,size=.05,color=black]{A5}
\Text[y=5,position=right,distance=1mm]{\small$u_2u_3$}
\Vertex[y=4,size=.05,color=black]{A4}
\Text[x=-2,y=4,position=left,distance=1mm]{\tiny$4$}
\Text[y=4,position=right,distance=4mm]{\small$u_2^2$}
\Vertex[y=3,size=.05,color=black]{A3}

\Text[y=3,position=right,distance=1mm]{\small$u_3$}
\Vertex[y=2,size=.05,color=black]{A2}
\Text[x=-2,y=2,position=left,distance=1mm]{\tiny$2$}
\Text[y=2,position=right,distance=1mm]{\small$u_2$}
\Vertex[y=0,size=.05,color=black]{A0}
\Text[x=-2,y=0,position=left,distance=1mm]{\tiny$0$}
\Text[y=0,position=right,distance=1mm]{\small$1$}

\Edge[lw=0.75pt](A5)(A6)
\Edge[lw=0.75pt,bend=45](A4)(A6)
\Edge[lw=0.75pt,bend=45,label={$\Sq^4$},position=left](A4)(A8)
\Edge[lw=0.75pt,bend=-45](A3)(A5)
\Edge[lw=0.75pt,bend=45,label={$\Sq^2$},position=left](A2)(A4)
\Edge[lw=0.75pt,label={$\Sq^1$},position=right](A2)(A3)
\end{tikzpicture}
\end{equation}

The suspension spectrum $J^8$ is 
well-defined $2$-locally but unstably 
it may not be. 

\begin{thm}\label{thm:JokerPD}
There is a simply connected Poincar\'e 
duality complex~$J^8$ which satisfies 
the following conditions.
\begin{itemize}
\item 
As an unstable algebra over the mod~$2$
Steenrod algebra,
\[
H^*(J^8;\F_2) = 
\F_2[u_2,u_3]/(u_2^3+u_3^2,u_2^2u_3) 
\]
with $\Sq^1u_2=u_3$ and all other Steenrod
square actions implied by the unstable 
condition.
\item 
As an $\Z[1/2]$-algebra,
\[
H^*(J^8;\Z[1/2]) = \Z[1/2][U_4]/(U_4^3).
\] 
\end{itemize}
\end{thm}
\begin{proof}
The basic idea is similar to the proof 
of theorem~\cite{AB:Joker}*{theorem~5.1}
and we merely sketch the details.

We start with $H^*(B\SO(3))=\F_2[w_2,w_3]$
and notice that the standard formula for 
Steenrod operations on Stiefl-Whitney classes 
$\Sq^1w_2=w_3$ and $\Sq^2w_3=w_2w_3$ imply
\[
\Sq^1(w_2^3+w_3^2)=w_2^2w_3,
\quad
\Sq^2(w_2^3+w_3^2)=w_2(w_2^3+w_3^2),
\quad
\Sq^4(w_2^3+w_3^2)=w_2^5
\]
hence $w_2^3+w_3^2$ and $w_2^2w_3$ generate 
an $\StA$-invariant ideal containing 
\[
w_2^5=w_2^2(w_2^3+w_3^2)+w_3(w_2^2w_3).
\] 

To construct this complex, first take 
a fibration $B\SO(3)\to K(\F_2,6)$ 
representing the cohomology class
$w_2^3+w_3^2\in H^6(B\SO(3))$. Then 
replace its fibre by a minimal CW 
complex and take the $8$-skeleton 
to obtain~$J^8$. 

Notice that $H^*(J^8;\F_2)$ is visibly 
a PD algebra and so is 
\[
H^*(J^8;\Z_{(2)}) = \Z_{(2)}[U_4,U_3]/(U_4^3,2U_3,U_4U_3,U_3^3).
\]
We leave the remaining details to 
the reader.
\end{proof}

The Spanier-Whitehead dual of $J^8_+$
is the Thom spectrum of the (stable)
Spivak normal bundle~$\nu$. So as left 
$\StA$-modules
\[
H^*(M\nu) \iso H_*(J^8)[8],
\]
i.e., $H_*(J^8)$ with a degree shift 
of~$8$. In fact with the usual Thom 
$H^*(J^8)$-module structure this is 
even an isomorphism of 
$H^*(J^8)\sharp\StA$-modules.

The dual Stiefel-Whitney classes here
are of course the Stiefel-Whitney 
classes of $\nu$ since for the Thom 
class $u\in H^0(M\nu)$ corresponding 
to the fundamental class $[J^8]\in H_8(J^8)$
(which in turn agrees with $[H^*(J^8)]$),
\[
\Sq^ru = \bar{w}_ru.
\]
The Stiefel-Whitney classes themselves 
are those of a putative tangent bundle 
of~$J^8$.

Here the Wu classes are zero except 
$v_4=u_2^2$. Therefore the non-zero
Stiefel-Whitney classes are
\begin{align*}
w_4 &= v_4 = u_2^2, \\
w_6 &= \Sq^2v_4 = u_3^2=u_2^3, \\
w_8 &= \Sq^4v_4 = u_2^4.
\end{align*}
By Proposition~\ref{prop:CharClasses}(c)
we have 
\begin{align*}
\bar{w}_4 &= u_2^2, \\
\bar{w}_6 &= u_2^3, \\
\bar{w}_8 &= 0.
\end{align*}
This leads to the following picture 
of the $H^*(J^8)\sharp\StA$-module
$H^*(M\nu)$.

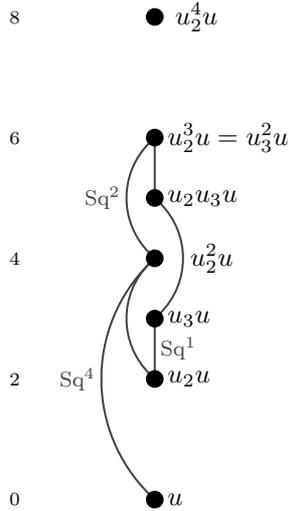
\bigskip
\begin{center}
\begin{tikzpicture}[scale=0.8]
\Vertex[y=8,size=.05,color=black]{A8}
\Text[x=-2,y=8,position=left,distance=1mm]{\tiny$8$}
\Text[y=8,position=right,distance=2mm]{\small$u_2^4u$}
\Vertex[y=6,size=.05,color=black]{A6}
\Text[x=-2,y=6,position=left,distance=1mm]{\tiny$6$}
\Text[y=6,position=right,distance=1mm]{\small$u_2^3u=u_3^2u$}
\Vertex[y=5,size=.05,color=black]{A5}
\Text[y=5,position=right,distance=1mm]{\small$u_2u_3u$}
\Vertex[y=4,size=.05,color=black]{A4}
\Text[x=-2,y=4,position=left,distance=1mm]{\tiny$4$}
\Text[y=4,position=right,distance=4mm]{\small$u_2^2u$}
\Vertex[y=3,size=.05,color=black]{A3}

\Text[y=3,position=right,distance=1mm]{\small$u_3u$}
\Vertex[y=2,size=.05,color=black]{A2}
\Text[x=-2,y=2,position=left,distance=1mm]{\tiny$2$}
\Text[y=2,position=right,distance=1mm]{\small$u_2u$}
\Vertex[y=0,size=.05,color=black]{A0}
\Text[x=-2,y=0,position=left,distance=1mm]{\tiny$0$}
\Text[y=0,position=right,distance=1mm]{\small$u$}

\Edge[lw=0.75pt](A5)(A6)
\Edge[lw=0.75pt,bend=45,label={$\Sq^2$},position=left](A4)(A6)
\Edge[lw=0.75pt,bend=-45](A3)(A5)
\Edge[lw=0.75pt,bend=45](A2)(A4)
\Edge[lw=0.75pt,label={$\Sq^1$},position=right](A2)(A3)
\Edge[lw=0.75pt,bend=45,label={$\Sq^4$},position=left](A0)(A4)
\end{tikzpicture}
\end{center}

\section{Obstruction theory for $J^8$}
\label{sec:J^8-ObstrThy}

We begin by recalling results of Madsen 
\& Milgram~\cite{IM&RJM:Book}*{theorem~7.1}
giving the $2$-local homotopy type of 
the classifying spaces for surgery. For 
non-experts we recommend the survey 
article of Klein~\cite{JRK:PDSpaces} 
for historical and mathematical 
background.
\begin{thm}\label{thm:M&Mthm7.1}
There are equivalences 
\begin{align*}
B(\G/\TOP)_{(2)} 
&\sim 
\prod_{j\geq1} K(\mathbb{Z}_{(2)},4j+1)\times K(\mathbb{F}_{(2)},4j-1), \\
B(\G/\PL)_{(2)} 
&\sim
E_3\times 
\prod_{j\geq2} K(\mathbb{Z}_{(2)},4j+1)\times K(\mathbb{F}_{(2)},4j-1),
\end{align*}
where $E_3$ is a $2$-stage Postnikov system 
with $k$-invariant 
$\beta\Sq^2\in H^6(K(\mathbb{F}_{(2)};3),\mathbb{Z})$.
\end{thm}

Now we state and prove our main result 
on $J^8$.

\begin{thm}\label{thm:J^8}
The Spivak normal bundle of $J^8$ is 
classified by a map $f\:J^8\to B\G$ 
such that the composition 
$J^8\to B\G\to B(\G/\PL)$ is null 
homotopic. Hence~$f$ factorises 
through a map $\tilde{f}\:J^8\to B\PL$,
and by composing with the natural map 
$B\PL\to B\TOP$ it also factors through 
a map $J^8\to B\TOP$.  
\[
\xymatrix{
& \ar@{.>}[dl]_{\tilde{f}}J^8\ar[d]^f\ar[dr]^{\sim0} & \\
B\PL\ar[r]\ar[d] & B\G\ar[r]\ar@{=}[d] & B(\G/\PL)\ar[d] \\
B\TOP\ar[r] & B\G\ar[r] & B(\G/\TOP)
}
\]
\end{thm}
\begin{proof}
We will use Theorem~\ref{thm:M&Mthm7.1}.
The odd degree integral cohomology of $J^8$ 
is trivial except for degree~$3$, so by 
inspection the only way a non-zero map 
$J^8\to B(\G/\PL)$ can occur is when 
the projection into the component~$E_3$ 
is nontrivial. 

Now there is a fibration sequence 
\[
K(\mathbb{Z},5)\to E_3 \to K(\mathbb{F}_2,3)
\]
and in the Serre spectral sequence the 
bottom generator $H^*(K(\mathbb{Z},5))$
transgresses to $\Sq^1\Sq^2z_3$ where 
$z_3\in H^3(K(\mathbb{F}_2,3))$ is the 
generator. Notice that
\begin{equation}\label{eq:ThmJ^8}
\Sq^1\Sq^2z_3 = \Sq^3z_3 = z_3^2.
\end{equation}
Since $H^5(J^8;\mathbb{Z})=0$, any 
non-trivial map $J^8\to E_3$ projects 
non-trivially to $\mathbb{F}_2$ and
so $z_3$ pulls back to $u_3$. But 
by \eqref{eq:ThmJ^8} this would 
imply that $u_3^2=0$ which is false. 

Therefore $[J^8,B\G/\PL]=0=[J^8,B\G/\TOP]$
and so the Spivak normal bundle lifts to 
a $\PL$-bundle at~$2$.
\end{proof}

\subsection*{The odd primary case}
Let's first consider the case of the prime~$3$.
Then 
\[
H^*(J^8;\Z_{(3)}) = \Z_{(3)}[U_4]/(U_4^3),
\quad
H^*(J^8;\F_3) = \F_3[u_4]/(u_4^3).
\]
However when we reduce modulo~$3$ we have 
to consider the Steenrod action. Tracing 
the construction back to $B\SO(3)$ we find 
that the generator~$u_4$ comes from the 
universal Pontrjagin class $p_1\in H^4(B\SO(3);\F_3)$
which in turn comes from $c_2\in H^4(B\SU(3);\F_3)$,
the reduction of the Chern class. A routine
calculation shows that 
$\StP^1c_2=c_2^2\in H^8(B\SU(3);\F_3)$ so 
$\StP^1u_4=u_4^2\in H^4(B\SU(3);\F_3)$.

Now according to~\cite{IM&RJM:Book}*{theorem~4.34(b)},
as spaces $\G/\PL[1/2]\sim B\O^\otimes[1/2]$. 
By a spectral sequence argument,
\[
H^*(B\G/\PL;\Z[1/2]) \iso H^*(B(B\O^\otimes);\Z[1/2])
\iso \Lambda_{B\O^\otimes}(e_{4k+1}:k\geq1).
\]
For our purposes it is enough to use the 
$8$-skeleton of a minimal CW structure, 
and
\[
H^*((B\G/\PL)^{[8]};\Z[1/2]) \iso 
\Lambda_{B\O^\otimes}(e_5:k\geq1).
\]
Thus at locally at $3$,
\[
(B\G/\PL)^{[8]}_{(3)} \sim K(\Z_{(3)},5)^{[8]}
\]
so $[J^8,(B\G/\PL)_{(3)}]=0$. For a prime 
$p>3$ a similar argument shows that 
$[J^8,(B\G/\PL)_{(p)}]=0$.

Therefore $[J^8,(B\G/\PL)[1/2]]=0$, and the 
Spivak normal bundle lifts to~$(B\PL)[1/2]$.

\section{Some calculations}\label{sec:Calculations}

Suppose that $X$ is a finite CW complex whose 
mod~$2$ cohomology is the unstable $\StA$-algebra 
which has the form shown in~\eqref{eq:H^*J8},
where the generators are $x_2,x_3$. We will 
determine $[X,B\SO]$, the group of isomorphism 
classes of stable orientable bundles on~$X$.
 Here $B\SO$ is the infinite loop space 
 $\Omega^{\infty}\kO\langle2\rangle$ so 
\[
[X,B\SO] \iso [\Sigma^\infty X,\kO\wedge\Sigma^2J]
\]
where $J$ is a CW spectrum whose cohomology 
as an $\StA(1)$-module is 
$H^*(J)\iso\StA(1)/\StA(1)\{\Sq^3\}$, i.e., 
has the form in~\eqref{eq:H^*J} desuspended 
twice (such spectra were shown to exist 
in~\cite{AB:Joker}). We can reinterpret 
the last group as 
$\mathscr{D}_{\kO}(\kO\wedge X,\kO\wedge\Sigma^2J)$, 
the mapping set in the $\kO$-module category 
$\mathscr{M}_{\kO}$ of~\cite{EKMM}. We can 
use Spanier-Whitehead duality to obtain
\[
\mathscr{D}_{\kO}(\kO\wedge X,\kO\wedge\Sigma^2J)
\iso
\mathscr{D}_{\kO}(\kO,\kO\wedge\Sigma^2J\wedge DX)
\]
where
\[
\kO\wedge\Sigma^2J\wedge DX \sim 
(\kO\wedge\Sigma^2J)\wedge_{\kO}(\kO\wedge DX)
\]
and as $\StA(1)$-modules
\[
H^*_{\kO}(\kO\wedge\Sigma^2J\wedge DX) 
\iso
H^*(\Sigma^2J)\otimes H^*(DX).
\]

Since $H^*_{\kO}(H)\iso\StA(1)$ there 
is an Adams spectral sequence
\begin{equation}\label{eq:ASS-kO}
\mathrm{E}_2^{s,t}=
\Ext_{\StA(1)}^{s,t}(H^*(\Sigma^2J)\otimes H^*(DX),\F_2)
\Lra 
\mathscr{D}_{\kO}(\kO,\kO\wedge\Sigma^{2+s-t}J\wedge DX).
\end{equation}
Of course we are interested in 
$\mathscr{D}_{\kO}(\kO,\kO\wedge\Sigma^{2}J\wedge DX)\iso [X,B\SO]$
which is computed from the terms $\mathrm{E}_2^{s,s}$
for~$s\geq0$. Now we are reduced to pure algebra. 

Notice that as $\StA(1)$-modules,
\begin{align*}
H^*(\Sigma^2J)\otimes H^*(DX) &\iso 
\StA(1)/\StA(1)\{\Sq^3\}\otimes\StA(1)/\StA(1)\{\Sq^3\}[-4]
\oplus \StA(1)/\StA(1)\{\Sq^3\}[-6] \\
&\iso 
\F_2 \oplus \StA(1)[-4]\oplus \StA(1)[-2]\oplus \StA(1)[-3]
\oplus \StA(1)/\StA(1)\{\Sq^3\}[-6].
\end{align*}
Here the $\StA(1)$-free summands contribute 
nothing in $\mathrm{E}_2^{s,s}$ while
\begin{align*}
\Ext^{s,s}_{\StA(1)}(\F_2,\F_2) &= \F_2\{h_0^s\}, \\
\Ext^{s,s}_{\StA(1)}(\StA(1)/\StA(1)\{\Sq^3\}[-6],\F_2) 
&= 
\Ext^{s,s+6}_{\StA(1)}(\StA(1)/\StA(1)\{\Sq^3\},\F_2) \\
&= \F_2\{h_0^sw\}
\end{align*}
where $w$ represents the Bott generator in 
$\pi_8(\kO\langle2\rangle)\iso\pi_8(\kO)$. 

It follows that 
\begin{equation}\label{eq:[X,BSO]}
[X,B\SO] \iso \Z_{(2)}\{\xi\}\oplus\pi_8(B\SO)
\end{equation}
where the second summand is pulled back from 
the top cell, and~$\xi$ restricted to the 
$3$-skeleton $X^{[3]}\simeq B\SO^{[3]}$ is 
induced by pulling back along the inclusion 
into~$B\SO$. The total Stiefel-Whitney class 
of~$\xi$ is
\[
w(\xi) = 1 + x_2 + x_3,
\]
and 
\[
w(2\xi) = 1 + x_2^2 + x_3^2.
\]
Therefore $2\xi$ is a candidate for a lift 
of the Spivak normal bundle of~$X$ 
from~$B\mathrm{G}$ to~$B\SO$.

\section{A homogeneous space realisation}\label{sec:Manifold}
The cohomology of a suitable manifold was 
determined by
Borel~\cite{ABo:CohomCertHomogSpces}*{theorem~13.7}
and Borel \& 
Hirzebruch~\cite{AB-&FH:CharClHomgSpces-I}*{section~17}.
This manifold also appears in work of Douglas et al
\cite{CLD-AGH-MAH:HomObstrStrOtns}.
\begin{thm}\label{thm:Manifold}
The exceptional Lie group\/ $\mathrm{G}_2$ contains 
a maximal subgroup of maximal rank isomorphic 
to\/~$\SO(4)$. Furthermore, the mod\/~$2$ 
cohomology of the homogeneous space 
$\mathrm{G}_2/\SO(4)$ satisfies
\[
H^*(\mathrm{G}_2/\SO(4)) \iso H^*(J^8)
\]
as unstable $\StA$-algebras.
\end{thm}
\begin{conj}\label{conj:Manifold}
There is a $2$-local homotopy equivalence 
$\mathrm{G}_2/\SO(4)\xrightarrow{\;\simeq\;}J^8$.
\end{conj}
%

In \cite{AB&TB:Joker}*{section~6} we realised 
the double version of the unstable Joker 
starting with $B\mathrm{G}_2$. There is 
a doubled version of~$J^8$ which is a 
$16$-dimensional Poincar\'e duality complex
$J^{16}$. Since the cohomology of this is 
concentrated in even degrees, this admits 
a $\PL$-structure. It seems an interesting 
question whether it admits a smooth 
structure and if this gives a $\String$-manifold.



\end{document}